\documentclass[ams, 11pt]{dgpaper}

\newcommand{\tenoreleven}[2]{#2} 

\pdfoutput=1

\title{Symplectic stability on manifolds with cylindrical ends}
\author{Sean Curry \; \'Alvaro Pelayo \; Xiudi Tang}

\hypersetup{pdftitle={Symplectic stability on manifolds with cylindrical ends}, pdfauthor={Sean Curry, \'Alvaro Pelayo, Xiudi Tang}, pdfsubject={Symplectic Geometry, Differential Geometry}}

\begin{document}

\begin{abstract}

A famous result of J\"urgen Moser states that a symplectic form on a compact manifold cannot be deformed within its cohomology class to an inequivalent symplectic form. It is well known that this does not hold in general for noncompact symplectic manifolds. The notion of Eliashberg-Gromov convex ends provides a natural restricted setting for the study of analogs of Moser's symplectic stability result in the noncompact case, and this has been significantly developed in work of Cieliebak-Eliashberg. Retaining the end structure on the underlying smooth manifold, but dropping the convexity and completeness assumptions on the symplectic forms at infinity we show that symplectic stability holds under a natural growth condition on the path of symplectic forms. The result can be straightforwardly applied as we show through explicit examples.

\end{abstract}
\subjclass[2010]{Primary 53D05; Secondary 53D35, 57R52, 58A14 }
\maketitle
\section{Introduction} \label{sec:introduction}

A fundamental problem in symplectic topology is that of determining when two symplectic forms are equivalent.
Recall the symplectic stability result of Moser~\cite{Mo1965} (1965) saying that if $\omega_t$, $t \in [0, 1]$, is a smooth path of cohomologous symplectic forms on a smooth manifold $M$ (i.e. an \emph{isotopy}), and $M$ is compact, then there exists a smooth path $\varphi_t$ of diffeomorphisms of $M$ such that $\varphi_t^* \omega_t = \omega_0$ (i.e. $\omega_t$, $t \in [0, 1]$, is a \emph{strong isotopy}). Moser's argument depends strongly on the assumption that $M$ is compact. The result does not generalize straightforwardly to the noncompact case. 
On the one hand, in his work on the h-principle Gromov showed that two cohomologous symplectic forms $\omega_0$ and $\omega_1$ on a noncompact manifold may be joined by an isotopy if and only if they are connected by a path of nodegenerate forms \cite{Gr69} (1969).
On $\mathbb{R}^{2n}$ the h-principle says that any two symplectic forms inducing the same orientation are isotopic.
On the other hand, in his paper on pseudoholomorphic curves \cite{Gr85} (1985) Gromov proved the existence of exotic symplectic structures on $\mathbb{R}^{2n}$, $n\geq 2$ (not symplectomorphic to the standard structure, though having the same orientation). See also \cite{BaPe, Mu}. 
In order to give a natural setting within which one may attempt to generalize stability and other results from compact to noncompact symplectic manifolds Eliashberg and Gromov \cite{ElGr} (1991) formalized the
notion of symplectic manifolds with convex ends, which has become a fundamental concept in symplectic topology. In particular
it led to important work of Cieliebak and Eliashberg, e.g., 
in their book on Stein and Weinstein manifolds~\cite{CiEl2012} where stability results are established for special classes of symplectic manifolds with convex ends, namely for Liouville manifolds and Weinstein manifolds.

Our goal is to drop the assumption that the symplectic forms be convex on the ends, keeping only the assumption that the underlying manifold has an end structure, i.e. can be viewed as the interior of a manifold with boundary. In order to do so, one must impose a growth condition on the path of symplectic forms, for which a metric is required. 
Recall that a Riemannian manifold $(M,g)$ \emph{has cylindrical ends} if there exists a compact codimension $0$ submanifold $K$ whose boundary $\partial K$ is a smooth hypersurface, and an isometry $M\setminus K \to \partial K \times (1,\infty)$ where $\partial K$ has the induced metric. The second 
component of the isometry may be smoothly extended to a function $M \to \mathbb{R}_+$ with values less than $1$ on $K^{\circ}$, referred to as the \emph{radial coordinate function} of $(M,g)$. The reciprocal of the radial coordinate is a defining function for the \emph{boundary at infinity} $\partial M$, diffeomorphic to $\partial K$. 
Let $\Abs{\cdot}_r$ denote the uniform norm with respect to the metric over the points with radial coordinate $r$. Let $S_a(M)$ be the set of symplectic forms on $M$ with cohomology class $a \in \Hml^2(M,\mathbb{R})$.
We define the \emph{log-variation} $\LogVar \colon S_a(M) \times \der\Omega^1(M) \to [0, \infty]$ by $$\LogVar(\omega, \beta) = \sup_{r \geq 1} r^{-1} \Abs{\omega^{-1}}_r \Abs{\vphantom{\omega^{-1}}\beta}_r.$$
Our main result gives a sufficient condition for symplectic stability on these manifolds.

\begin{maintheorem} \label{thm:moser-cylindrical-end}
   Let $M$ be a manifold with cylindrical ends and $\Hml^1(\partial M,\mathbb{R}) = 0$. If $\omega_t$, $t \in [0, 1]$, is a symplectic isotopy with total log-variation $$\int_0^1 \LogVar(\omega_t,\dot{\omega}_t)\, \der t < \infty$$ then it is a strong isotopy.
\end{maintheorem}

The condition in the theorem is not necessary, see \cref{bad-example}. It is natural, however, in the sense that it amounts to a natural growth condition on the size of the vector field $X_t$ constructed via a generalization of Moser's Path Method to the noncompact case (\cref{sec:noncompact-path-method}). The difficulty with establishing a necessary condition in terms of the growth of the family $\omega_t$ is that $X_t$ may grow rapidly at infinity, yet still be complete.

\begin{corollary} \label{thm:main-corollary}
  Let $M$ be a manifold with cylindrical ends and $\Hml^1(\partial M,\mathbb{R}) = 0$. Then a symplectic isotopy $\omega_t$, $t \in [0,1]$, is a strong isotopy if there exists $C > 0$ such that $ \Abs{\omega_t^{-1}}_r \Abs{\vphantom{\omega_t^{-1}}\dot{\omega}_t}_r \leq C r$  for $r\gg 0, t\in [0,1]$.
\end{corollary}

\begin{corollary} \label{cor:symplectic-metric}
  Let $M$ be a manifold with cylindrical ends and $\Hml^1(\partial M,\mathbb{R}) = 0$, and fix $a \in \Hml^{2}(M,\mathbb{R})$. Then $S_a(M) \times S_a(M) \to [0,\infty]$ given by
$(\alpha,\beta) \mapsto \inf \big(\int_0^1 \LogVar(\omega_t,  \dot{\omega}_t)\der t\big)$,
where the infimum is taken over all isotopies from $\alpha$ to $\beta$, is a pseudometric.
Moreover, forms at finite distance are strongly isotopic.
\end{corollary}

Corollaries \ref{thm:main-corollary} and \ref{cor:symplectic-metric}  follow immediately from the Main Theorem.  

\begin{corollary} \label{cor:moser-linear-family}
  Let $M$ be a manifold with cylindrical ends and radial coordinate function $\textbf{\emph{r}}$, with $\Hml^1(\partial M, \mathbb{R}) = 0$.
  Let $\omega$ be a symplectic form and $\sigma$ a $1$-form on $M$.
  Suppose that $\sup_{r \in \textbf{\emph{r}}(M)} \Abs{\omega^{-1}}_r \Abs{\vphantom{\omega^{-1}} \der \sigma}_r < 1$.
  Then $\omega + t \der \sigma$, $t \in [0, 1]$, is a strong isotopy of symplectic forms.
\end{corollary}

\begin{corollary} \label{cor:moser-euclidean}
  A symplectic isotopy $\omega_t$, $t \in [0,1]$, on $\mathbb{R}^{2n}$, $2n\geq 4$, is a strong isotopy if
  there exists $C>0$ such that $\Abs{\omega_t^{-1}}_r \Abs{\vphantom{\omega_t^{-1}}\dot{\omega}_t}_r \leq C \log r$
for $r\gg1, t \in [0,1]$,  where $\|\cdot\|_r$ is the uniform 
Euclidean norm over the sphere of radius $r$.
\end{corollary}

\begin{corollary} \label{cor:symplecic-compact-punctured}
  Let $M$ be an even dimensional compact manifold, $\mathrm{dim}M\geq 4$, and let $F$ be a finite set of points on $M$.
  If $\omega_t$, $t \in [0,1]$, is a symplectic isotopy on $M \setminus F$ for which $\omega_t^{-1}$ and $\dot{\omega}_t$ are bounded uniformly in $t$ with respect to any fixed metric on $M$, then $\omega_t$ is a strong isotopy on $M\setminus F$.
\end{corollary}

The Main Theorem, and Corollaries \ref{cor:moser-linear-family} and \ref{cor:symplecic-compact-punctured} are proved in \cref{sec:moser-cylindrical-end}. \cref{cor:moser-euclidean} follows by noting that, away from the origin, Euclidean space is conformal to a cylinder.


\begin{remark}
The assumption $\Hml^1(\partial M,\mathbb{R}) = 0$ is equivalent to the natural map $\Hml^2_c(M,\mathbb{R}) \to \Hml^2(M,\mathbb{R})$ being injective. This allows one to handle the compact part of $M$ separately in constructing the generator $X_t$ of the strong isotopy via the Path Method (\cref{sec:noncompact-path-method}). More importantly, this assumption implies injectivity of the map $\Hml^2_\cspt(V,\mathbb{R})\to\Hml^2(V,\mathbb{R})$ for sets $V$ of the form $\textbf{\emph{r}}^{-1}(r-\epsilon,r+\epsilon)$, where $\textbf{\emph{r}}$ is the radial coordinate function of $M$. Without this assumption it is impossible to construct the time dependent vector field $X_t$ with bounds on $X_t$ which are localized in the radial coordinate on the ends. This makes the assumption natural, and apparently necessary for our kind of results. The assumption $\Hml^1(\partial M,\mathbb{R}) = 0$ also implies $\dim M > 2$. If $\dim M=2$ a symplectic isotopy is a strong isotopy if $\int_M\omega_0=\int_M\omega_1$ and the set of ends where $\omega_0$ and $\omega_1$ give infinite volume coincide up to permutation by a diffeomorphism \cite{GrSh1979,PeTa2016}. 
\end{remark}

\begin{remark}
The role of the condition $ \Abs{\omega_t^{-1}}_r \Abs{\vphantom{\omega_t^{-1}}\dot{\omega}_t}_r \leq C r$ in \cref{thm:main-corollary} is intuitive and natural: it prevents finite time blow up for the ordinary differential inequality of the form $\dot{r}(t)\leq\Abs{\omega_t^{-1}}_r \Abs{\vphantom{\omega_t^{-1}}\dot{\omega}_t}_r$. Heuristically, this inequality controls the escape to infinity of the integral curves for the time dependent vector field $X_t$, constructed by generalizing Moser's Path Method (\cref{sec:noncompact-path-method}), whose flow gives the strong isotopy. In practice, one only obtains an inequality of (roughly) this form for $r$ in a set of intervals with arbitrarily small gaps between them, which makes formalizing this heuristic argument awkward. Our approach, therefore, is to control the lengths of the integral curves more directly, leading to the result obtained in the Main Theorem. 
\end{remark}

\begin{remark}
In stating our Main Theorem and some of its corollaries we have made use of a Riemannian metric with cylindrical ends. This metric plays only an auxiliary role, allowing us to give the simplest formulation of our result. Metrics with different asymptotics can be used. This is demonstrated for the most basic case of the Euclidean metric in \cref{cor:moser-euclidean}. For concrete examples our conditions are also very easy to check. The following is a simple application of \cref{cor:moser-euclidean}:
If  $f_1,f_2$ are smooth functions bounded away from zero and with bounded time derivative and $c$ is any constant, then the isotopy of symplectic forms
$\omega_t= f_1(t,x_1,y_1)\der x_1\wedge \der y_1+f_2(t,x_2,y_2)\der x_2\wedge \der y_2 + c \der x_1\wedge \der x_2$, $t\in [0,1]$, on $\mathbb{R}^4$ is a strong isotopy. More generally, the time derivatives of $f_1$ and $f_2$ may have logarithmic growth in $r$, the radial coordinate on $\mathbb{R}^4$.
\end{remark}

\subsubsection*{Acknowledgments:} The last two authors are supported by NSF CAREER Grant DMS-1518420. We are very grateful to Roger Casals, Daniel Cristofaro-Gardiner, Yakov Eliashberg, Larry Guth, Rafe Mazzeo, Leonid Polterovich, Justin Roberts, Alan Weinstein, Paul Yang, and Shing-Tung Yau for helpful discussions about symplectic stability.

\section{Path Method on noncompact manifolds} \label{sec:noncompact-path-method}

For $M$ compact, Moser proved his symplectic stability result by
differentiating $\varphi_t^* \omega_t = \omega_0$ to get
$0 = \frac{\der}{\der t}(\varphi_t^* \omega_t) = \varphi^*_t \Pa{\dot{\omega}_t + \mathcal{L}_{X_t} \omega_t}$,
where $\dot{\omega}_t$ is the time derivative of $\omega_t$ and $X_t$ is the time-dependent vector field generating the family $\varphi_t$, and then solving for $\varphi_t$ in terms of $X_t$.
Since $[\omega_t]$ is constant, $\dot{\omega}_t$ is exact for all $t \in [0, 1]$.
By Hodge theory on compact manifolds there exists a smooth family $\sigma_t$ of $1$-forms such that $\dot{\omega}_t = \der \sigma_t$ for all $t \in [0, 1]$.
By Cartan's formula $\mathcal{L}_{X_t}\omega_t = \der(X_t \intprod \omega_t)$ since $\omega_t$ is closed for each $t \in [0, 1]$.
So $\dot{\omega}_t + \mathcal{L}_{X_t} \omega_t= \der(\sigma_t + X_t \intprod \omega_t)$.
If one chooses $X_t$ to be the vector field determined by $\sigma_t + X_t \intprod \omega_t = 0$ then, since $M$ is compact, we may integrate $X_t$ to determine a family $\varphi_t$ such that $\varphi_t^* \omega_t = \omega_0$ for all $t \in [0, 1]$. 
This technique is usually called the \emph{Path Method}. In the noncompact case, the argument above does not work, and the conclusion is false. 
The problem lies in being able to solve $\dot{\omega}_t = \der \sigma_t$ for a smooth family of $1$-forms $\sigma_t$ in such a way that $X_t$, $t\in [0,1]$, is complete.

The following is the outline of the steps we carry out to construct the vector field $X_t$ and provide the $\Leb^\infty$ estimates needed to determine the existence of the flow when $M$ is not compact: In the first step we consider a compact Riemannian manifold $(N, \gN)$ of dimension $m$ and an open interval $J$.
Combining Hodge theory on $(N, \gN)$ with the Poincar\'e Lemma one has, for any $k$ with $1 \leq k \leq m$, an operator $I^k_{N \times J} \colon \Omega^k(N \times J) \to \Omega^{k-1}(N \times J)$  satisfying $\der I^k_{N \times J} \omega = \omega$ for all $\omega \in \der \Omega^{k-1}(N \times J)$.
We bound the $\Leb^\infty$ norm of $I^k_{N \times J}$ by proving (for $m \geq 3$) that $I^k_N = \der^* \circ G \colon \Omega^k(N) \to \Omega^{k-1}(N)$ has finite $\Leb^\infty$ norm, where $G$ is the Green's operator for the Hodge Laplacian on $k$-forms and $\der^*$ is the codifferential.

In the second step we solve the $\der$-equation for compactly supported forms.
Let $M$ be a smooth manifold and let $V$ be an open submanifold of $M$ with compact closure and smooth boundary.
We use the weighted Hodge theory of Bueler-Prohorenkov \cite{BuPr2002} on noncompact manifolds to construct an operator $I^k_{M, V} \colon \Omega^k_\cspt(M, V) \to \Omega^{k-1}(M, V)$ on forms compactly supported in $V$ satisfying $\der \circ I^k_{M, V}\omega = \omega$ for all $\omega \in \der \Omega^{k-1}_\cspt(M, V)$.

In the final step, given an isotopy of symplectic forms $\omega_t$, $t \in [0, 1]$, we put the previous steps together to construct a time-dependent vector field $X_t$ satisfying $\der(X_t \intprod \omega_t) = -\dot{\omega}_t$ with explicit $\Leb^\infty$ estimates in terms of the $\Leb^\infty$ norms of $\dot{\omega_t}$, $\omega_t^{-1}$, and the operators $I^k_{N \times J}$ and $I^k_{M, V}$ for a collection of precompact pieces $U\cong N \times J$ and $V$ of the underlying manifold $M$.
To define these pieces we pick a proper smooth function $f$ and a covering of $f(M)$ by intervals whose preimages give the sets $U$ and $V$.
For intervals $J$ not containing any critical values of $f$ we identify $U = f^{-1}(J)$ with $N \times J$, where $N = f^{-1}(r_0)$ for some $r_0 \in J$, and define $\sigma_t = I^2_{N \times J} \dot{\omega}_t$, for which we have explicit $\Leb^\infty$ estimates from the first step.
We then smoothly extend $\sigma_t$ across the remaining gluing regions, corresponding to the remaining intervals $J'$, to solve $\der\sigma_t = \dot{\omega}_t$.
This requires using the operator $I^2_{M, V}$ from the second step with $V = f^{-1}(J')$. This gluing step is topologically obstructed, and we must assume that $\Hml^2_\cspt(V,\mathbb{R})\to\Hml^2(V,\mathbb{R})$ is injective (this is the reason for the condition $\Hml^1(\partial M,\mathbb{R})=0$ in our Main Theorem). We then let $X_t=-\omega_t^{-1}\sigma_t$. Since $\der(X_t \intprod \omega_t) = -\dot{\omega}_t$, the local flow of $\varphi_t$ of $X_t$ starting from $t_0 = 0$ satisfies $\frac{\der}{\der t}(\varphi_t^* \omega_t) = 0$, where this makes sense.
So the problem reduces to studying the global existence of the flow $\varphi_t$ for $t \in [0, 1]$.
This is done in \cref{sec:moser-cylindrical-end} using the precise estimates on $X_t$ which appear in \cref{lem:noncompact-path-method}.

\subsection*{Step 1: \texorpdfstring{$\Leb^\infty$}{L\^{}infinity} estimates for solving the \texorpdfstring{$\der$}{d}-equation} \label{ssec:poincare-lemma}

Let $N$ be a smooth manifold and $J$ an open interval.
The Poincar\'e Lemma for de Rham cohomology states that $\Hml^k(N \times J,\mathbb{R}) = \Hml^k(N,\mathbb{R})$ for any $k$.
This is proved by fixing any $r_0 \in J$ and constructing a de Rham homotopy operator for the pair of maps $\pi \colon N \times J \to N$, the projection, and $\iota \colon N \hookrightarrow N \times J$, the inclusion $y \mapsto (y, r_0)$.
An example of such a homotopy operator is the map $I^k_0 \colon \Omega^k(N \times J) \to \Omega^{k-1}(N \times J)$ given by $(I^k_0 \omega)(y, r) = \int_{r_0}^r \partial_s \intprod \omega(y, s) \der s$ for each $(y, r) \in N \times J$, where $\partial_s$ is the coordinate vector field along $J$.
A straightforward calculation shows that $\der I^k_0 \omega + I^k_0 \der \omega = \omega - \pi^*\iota^*\omega$ for any $\omega \in \Omega^k(N \times J)$, with $0 \leq k \leq \dim N+1$.
We will make use of the following trivial consequence.

\begin{lemma} \label{lem:primitive-cylinder}
  Let $N$ be a smooth manifold and $J$ an open interval.
  Let $k \in \Set{1, \ldots, \dim N}$ and let $I^k_N \colon \Omega^k(N) \to \Omega^{k-1}(N)$ be a smooth operator such that $\der I^k_N = \identity$ on $\der \Omega^{k-1}(N)$.
  Fix $r_0 \in J$ and let $\iota \colon N \hookrightarrow N \times J$ be the map $y \mapsto (y, r_0)$.
  Then the operator $I^k_{N \times J} \colon \Omega^k(N \times J) \to \Omega^{k-1}(N \times J)$ given by 
  \begin{equation*}\label{eq:I_U}
    (I^k_{N \times J} \omega)(y, r) = \int_{r_0}^r \partial_s \intprod \omega(y, s) \der s + (I^k_N \iota^* \omega)(y)
  \end{equation*}
  satisfies $\der I^k_{N \times J} \omega = \omega$ for all $\omega \in \der \Omega^{k-1}(N \times J)$.
\end{lemma}

We will be applying \cref{lem:primitive-cylinder} in the case of a compact Riemannian manifold $(N, \gN)$.
In order to bound the $\Leb^\infty$ norm of $I^k_{N \times J}$ it suffices to prove that the natural Hodge theoretic operator $I^k_N$ has finite $\Leb^\infty$ norm.

\begin{theorem} \label{lem:primitive-compact-estimate}
  Let $(N, \gN)$ be a compact Riemannian manifold of dimension $m \geq 3$.
  Let $k \in \Set{1, \dotsc, m}$ and let $I^k_N = \der^* \circ G \colon \Omega^k(N) \to \Omega^{k-1}(N)$ where $G$ is the Green's operator for the Hodge Laplacian on $k$-forms, and $\der^*$ is the codifferential.
  Then $\der \circ I^k_N$ is the identity on $\der \Omega^{k-1}(N)$ and
  \begin{equation*}
    \Abs{I^k_N}_{\Leb^\infty} = \sup_{\omega \in \Omega^k(N)} \frac{\Abs{I^k_N \omega}_{\Leb^\infty(N, \gN)}}{\Abs{\omega}_{\Leb^\infty(N, \gN)}} < \infty.
  \end{equation*}
  Here $\Abs{\cdot}_{\Leb^\infty(N, \gN)}$ is the uniform norm with respect to $g_N$ over $N$.
\end{theorem}

\begin{proof}
  The Green's operator $G \colon \Omega^k(N) \to \Omega^k(N)$ is characterized by $\Delta G \omega = \omega$ for $\omega \in (\ker \Delta)^\perp$ and $G \omega = 0$ for $\omega \in \ker \Delta$, where $\Delta \colon \Omega^k(N) \to \Omega^k(N)$ is the Hodge Laplacian.
  It is possible to construct an integral kernel for $G$; the only difficulty is that the Green's kernel must be thought of as a distributional section of the bundle $\pi_1^* \Lambda^2 N \otimes \pi_2^* (\Lambda^2 N)^* \to N \times N$ where $\pi_1, \pi_2 \colon N \times N \to N$ are the projections onto the first and second factor respectively.
  We will show that the Green's kernel has the same asymptotic behavior at leading order near the diagonal as the Euclidean Green's function (cf. \cite{Au1998} for the case of functions).
  To construct the Green's kernel we solve $\Delta_{q, \; \text{distr.}} G(p, q) = \delta_p(q) - V^{-1}$  where $\Delta_{q, \; \text{distr.}}$ is the distributional Laplacian, $\delta_p(q)$ is the Dirac delta function at $p$, and $V$ is the volume of $(N, \gN)$.
  We start by formally approximating $G(p, q)$ near the diagonal.
  Let $f \in \Cont_0^\infty(\mathbb{R})$ be the standard bump function equal to $1$ on $(-\frac{\delta}{2}, \frac{\delta}{2})$ and supported in $(-\delta, \delta)$ where $\delta$ is the injectivity radius of $(N, \gN)$.
  Let $H(p, q) = \frac{\operatorname{dist}(p, q)^{2-m}}{(m-2) \sigma_{m-1}} f(\operatorname{dist}(p, q))$ where $\sigma_{m-1}$ is the volume of the $(m-1)$-sphere.
  Let $n$ be an integer larger than $\frac{m}{2}$.
  Let $\Gamma_1(p, q) = -\Delta_q H(p, q)$ and for $1 \leq i \leq n$ let $\Gamma_{i+1}(p, q) = -\int_N \Gamma_i(p, r) \Delta_q H(r, q) \operatorname{dvol}_q$.
  We write
  \begin{equation*} 
    G(p, q) = H(p, q) + \sum_{i =1}^n \int_N \Gamma_i(p, r) H(r, q) \operatorname{dvol}_q + F(p, q)
  \end{equation*}
  where $F(p, q)$ is a distributional section of $\pi_1^* \Lambda^2 N\otimes \pi_2^* (\Lambda^2 N)^* \to N \times N$, and seek to solve for $F(p,q)$.
  Taking the Laplacian of $G(p,q)$, using that $\Delta_{q, \; \text{distr.}} H(p, q) = \Delta_q H(p, q) + \delta_p(q)$ by Green's third identity (see for instance p. 107 in \cite{Au1998}), and canceling,
  \begin{equation} \label{eq:boundedRHS}
    V^{-1} = \Gamma_{n+1}(p, q) + \Delta_{q, \; \text{distr.}} F(p, q).
  \end{equation}
  By a standard Lemma of Giraud \cite[p. 150]{Gi1929} $\Gamma_n(p, q)$ is bounded, and consequently $\Gamma_{n+1}(p, q)$ is $\Cont^1$.
  By elliptic theory, for each fixed $p$ there is a weak solution $F(p, q)$ of \cref{eq:boundedRHS}.
  Then by elliptic regularity for elliptic operators between vector bundles whose principal part has scalar coefficients the solution $F(p, q)$ is $\Cont^2$.
  It follows from the definition of $H(p,q)$ and the ansatz for $G(p,q)$ above that $G(p, q) = \frac{\operatorname{dist}(p, q)^{2-m}}{(m-2) \sigma_{m-1}} (1 + \mathcal{O}(\operatorname{dist}(p, q)))$ near the diagonal.
  Thus $\big|\int_{B_\delta(p)} \der^*_p G(p, q) \omega(q) \operatorname{dvol}_q\big|$ is at most
  \begin{equation*}
    \abs{\int_{B_\delta(p)} \frac{r^{1-m}}{(m-2) \sigma_{m-1}} \Pa{1 + \mathcal{O}(r)} \, r^{m-1} \der r \operatorname{dvol}_{\mathbb{S}^{m-1}}} \Abs{\omega}_{\Leb^\infty(N, g)}
  \end{equation*}
  where $r = \operatorname{dist}(p, q)$.
  Since the derivative of $G(p, q)$ is bounded outside of the ball $B_\delta(p)$ and $N$ is compact there exists $C>0$ for which
  \begin{align*}
    \abs{(I^k_N \omega)(p)}_g 
    &= \abs{\int_N \der^*_p G(p, q) \omega(q) \operatorname{dvol}_q}_g \\
    &\leq \abs{\int_{B_\delta(p)} \der^*_p G(p, q) \omega(q) \operatorname{dvol}_q}_g + \abs{\int_{N \setminus B_\delta(p)} \der^*_p G(p, q) \omega(q) \operatorname{dvol}_q}_g \\ 
    &\leq C \Abs{\omega}_{\Leb^\infty(N, g)}
  \end{align*}
  for all $p \in N$.
\end{proof}

\subsection*{Step 2: Solving the \texorpdfstring{$\der$}{d}-equation for compactly supported forms} \label{ssec:compact-support}

\begin{lemma} \label{lem:primitive-form-compact-support}
  Let $M$ be a smooth manifold and let $V$ be an open submanifold of $M$ with smooth compact boundary. Let $\Omega^k(M, V)$ be the space of $k$-forms which vanish outside of $V$.
  For $k \in \Set{1, \ldots, \dim M}$ there exists a smooth operator $I^k_{M, V} \colon \Omega^k_\cspt(V) \to \Omega^{k-1}(M, V)$ such that $\Res{\Pa{\der \circ I^k_{M, V}\omega}}_V=\omega$ for all $\omega\in\der \Omega^{k-1}_\cspt(V)$.
\end{lemma}

\begin{proof}
  Let $g_N$ be a metric on $N=\partial V$.
  Let $U$ be a tubular neighborhood of $N$ and let $\rho$ be a defining function for $N$ such that $U=\rho^{-1}(-1,1)$ and $\rho>0$ on $V$.
  Fix a diffeomorphism $U\to N\times (-1,1)$ with the second component being $\rho$. Let $f=\rho^{-1}$ on $V$ and use this diffeomorphism to identify $U\cap V$ with $N\times (1,\infty)$.
  The metric $g_N \oplus dr^2$ on $N\times (1,\infty)$ may be extended to a complete metric $g_V$ on $V$. Let $\mathcal{S}(\Lambda^k V)$ be the space of smooth $k$-forms $\omega$ on $V$ with rapid decay in the sense that $\lim_{r\to\infty}|f^\ell \partial^{\alpha}\omega|(y,r) =0$ for any multiindex $\alpha$, $\ell\in \mathbb{N}$, and choice of local coordinates on $N$ (here the coordinate derivatives are with respect to $(y,r)$ and act only on the coefficients of the differential form). 
  A $k$-form $\omega$ in $e^{-2f^2}\mathcal{S}(\Lambda^k V)$ vanishes to infinite order on $N=\partial V$, and thus extends smoothly by zero to all of $M$.
  Let $\mu=e^{2f^2}\mathrm{dvol}_{g_{V}}$ where $\mathrm{dvol}_{g_{V}}$ is the Riemann-Lebesgue measure. Then $\der^*_\mu = e^{-2f^2}\der^*e^{2f^2}$ is the formal adjoint of $\der$ with respect to $\mu$. Let $\Delta_{\mu}= \der \der^*_\mu + \der^*_\mu\der$. 
  By the Hodge decomposition of \cite{BuPr2002} there exists a Green's operator $G_\mu$ for $\Delta_\mu$ with domain and codomain equal to $e^{-2f^2}\mathcal{S}(\Lambda^k V)$, which properly contains $\Omega^k_\cspt(V)$. 
  By definition we then have $\der \der_{\mu}^*G_\mu\omega=\omega$ for all $\omega\in \der \Omega^{k-1}_\cspt(V)$, and we define $I^k_{M, V}$ to be $\der_{\mu}^*G_\mu$ composed with extension by zero.
\end{proof}

\subsection*{Step 3: Piecewise construction of \texorpdfstring{$X_t$}{X\_t} with estimates} \label{ssec:moser-noncompact-prep}

Given a compact Riemannian manifold $(N, \gN)$ we denote, as in \cref{lem:primitive-compact-estimate}, the Hodge theoretic right inverse to the exterior derivative $\der:\Omega^1(N) \to\Omega^2(N)$ by $I^2_N$.
\begin{lemma} \label{lem:noncompact-path-method} 
  Let $M$ be a manifold, $\dim M \geq 4$, and $f \colon M \to \mathbb{R}$ an exhaustion.
  Let $[a_i, b_i]$, $i \in \mathbb{N}$, be intervals containing no critical values of $f$ such that $a_i < b_i < a_{i+1}$ and $b_i \rightarrow \infty$. Let $X = \cup_{i\in \mathbb{N}}[a_i,b_i]$ and suppose that $\Hml^2_\cspt(M \setminus f^{-1}(X),\mathbb{R}) \to H^2(M \setminus f^{-1}(X),\mathbb{R})$ is injective.
  Let $g$ be a Riemannian metric on $M$ such that $\nabla f$ is a unit Killing vector field on $f^{-1}(X)$.
  If $\omega_t$, $t\in[0,1]$, is an isotopy of symplectic forms on $M$, then there exists a time-dependent vector field $X_t$ on $M$, $t\in[0,1]$, satisfying $\der(X_t \intprod \omega_t) = -\dot{\omega}_t$ and on each $U_i =f^{-1}(a_i, b_i)$ 
  \begin{equation*}
    \Abs{X_t}_{\Leb^\infty(U_i, g)}\leq \Pa{\frac{b_i-a_i}{2}+\Abs{I^2_{f^{-1}(\frac{a_i+b_i}{2})}}_{\Leb^\infty}}\Abs{\omega_t^{-1}}_{\vphantom{\overset{.}{L}}\Leb^\infty(U_i, g)}\Abs{\vphantom{\omega_t^{-1}}\dot{\omega}_t}_{\vphantom{\overset{.}{L}}\Leb^\infty(U_i, g)}.
  \end{equation*}
\end{lemma}

\begin{proof}
  For each $i \in \mathbb{N}$ let $J_i = (a_i, b_i)$ and choose enlarged intervals $\tilde{J}_i = (\tilde{a}_i,\tilde{b_i})$ such that the closures $[\tilde{a}_i, \tilde{b_i}]$ do not contain critical points of $f$, and $\tilde{a}_i < a_i < b_i < \tilde{b_i} < \tilde{a}_{i+1}$ for all $i \in \mathbb{N}$.
  For each $i \in \mathbb{N}$ let $\hat{J}_i = (\frac{a_i+2\tilde{a}_i}{3}, \frac{b_i+2\tilde{b}_i}{3})$, so that $J_i \subsetneq \hat{J}_i \subsetneq \tilde{J}_i$, and let $U_i = f^{-1}(J_i)$, $\hat{U}_i = f^{-1}(\hat{J}_i)$, and  $\tilde{U}_i = f^{-1}(\tilde{J}_i)$.
  Let $r_i =\frac{a_i+b_i}{2}$, and let $\iota_{r_i} \colon f^{-1}(r_i) \to M$ be the inclusion.
  Using the flow $\psi$ of $\nabla f$ we may identify $\tilde{U}_i$ with $f^{-1}(r_i) \times \tilde{J}_i$.
  We thus define the $1$-form $\sigma^i_t$ on $\tilde{U}_i$ by
  \begin{equation}\label{eq:sigma_i}
    \sigma^i_t(y, r) = \int_{r_i}^r \nabla f \intprod \psi_{s - r}^* \Pa{\dot{\omega}_t (y, s)} \der s + \big(I^2_{f^{-1}(r_i)} \iota_{r_i}^* \dot{\omega}_t\big)(y),
  \end{equation}
  for $(y, r) \in f^{-1}(r_i) \times \tilde{J}_i$.
  By \cref{lem:primitive-cylinder} we have $\der \sigma^i_t = \dot{\omega}_t$ on $\tilde{U}_i$.
  Let $\lambda_i \colon M \to [0, 1]$ be a smooth function supported in $\tilde{U}_i$ and equal to $1$ in a neighborhood of $\hat{U}_i$.
  Let $\alpha_t = \dot{\omega}_t - \sum_{i=1}^\infty \der(\lambda_i \sigma^i_t) = -\sum_{i=1}^\infty \der\lambda_i \wedge \sigma^i_t + \big(1 - \sum_{i=1}^\infty \lambda_i\big) \dot{\omega}_t$.
  Let $J_{0,1}= (-\infty,\frac{2a_{1}+\tilde{a}_{1}}{3})$ and for each $i \in \mathbb{N}$ let $J_{i, i+1}= (\frac{2b_i+\tilde{b}_i}{3}, \frac{2a_{i+1}+\tilde{a}_{i+1}}{3})$.
  For all $i \in \mathbb{N}\cup\{0\}$ let $V_{i, i+1} = f^{-1}(J_{i, i+1})$.
  Note that $\alpha_t$ is supported in the union of the gluing regions $V_{i, i+1}$, moreover $\Res{\alpha_t}_{V_{i, i+1}}$ is compactly supported in $V_{i, i+1}$.
  Since $\alpha_t$ is exact on $M$ and, by assumption, $\Hml^2_\cspt(V_{i, i+1}) \to\Hml^2(V_{i, i+1})$ is injective we have $[\alpha_t|_{V_{i, i+1}}]_{\Hml^2_\cspt(V_{i, i+1})}= 0$.
  By \cref{lem:primitive-form-compact-support}, there is $\beta^{i, i+1}_t \in \Omega^1(M)$ which vanishes outside $V_{i, i+1}$, and satisfies $\der \beta^{i, i+1}_t = \alpha_t$ on $V_{i, i+1}$.
  Let $\beta_t = \sum_{i = 0}^\infty \beta^{i, i+1}_t$ and let $\sigma_t = \sum_{i=1}^\infty (\lambda_i \sigma^i_t) + \beta_t$.
  Then $\dot{\omega}_t = \der \sigma_t$.
  Hence the time-dependent vector field given by $X_t = -\omega_t^{-1} \sigma_t$ for $t \in [0, 1]$ satisfies $\der(X_t \intprod \omega_t) = -\dot{\omega}_t$.
  The estimate follows from \cref{eq:sigma_i}.
\end{proof}

\section{Symplectic stability on manifolds with cylindrical ends} \label{sec:moser-cylindrical-end}

\begin{lemma} \label{lem:escape-lemma}
  Let $M$ be a smooth manifold and $X_t$, $t \in [0, 1]$, a smooth time-dependent vector field on $M$.
  Let $\gamma \colon J \to M$ be the maximal flow line of $X_t$ with $\gamma(0) = x_0$.
  If $\gamma(J)$ is contained in a compact set then $J = [0, 1]$.
\end{lemma}

\begin{proof}
  Suppose that $J \neq [0, 1]$, then there is $T \in (0, 1]$ such that $J = [0, T)$.
  Define $\tilde X$ on $M \times [0, 1]$ by $\tilde X = X_t + \partial_t$, and let $\tilde \gamma$ be the maximal integral curve of $\tilde X$ with $\tilde{\gamma}(0) = (x_0, 0)$.
  Then $\tilde \gamma$ has maximal domain $J$ and is given by $\tilde{\gamma}(t) = (\gamma(t), t)$.
  By the standard Escape Lemma \cite[Lemma 9.19]{MR2954043} $\tilde{\gamma}(J)$ is not contained in any compact subset of $M \times [0, 1]$.
  But this implies that $\gamma(J)$ is not contained in any compact subset of $M$.
\end{proof}

\begin{lemma} \label{lem:moser-noncompact-exhausted}
  Suppose that $M$ is a noncompact manifold and $f \colon M \to \mathbb{R}$ is an exhaustion such that $\Hml^1(f^{-1}(r),\mathbb{R}) = 0$ for $r > R$.
  Let $\Set{r_i}_{i \in \mathbb{N}} \subset \mathbb{R}_{>R}$, $\Set{\delta_i}_{i \in \mathbb{N}}\subset \mathbb{R}_{>0}$, and $\Set{\alpha_i}_{i \in \mathbb{N}} \subset \mathbb{R}_{\geq 0}$ be sequences such that the intervals $[r_i-\delta_i, r_i + \delta_i]$ are disjoint and contain no critical values of $f$, and $\sum_{i=1}^\infty \alpha_i\delta_i = \infty$.
  Let $g$ be a metric on $M$ which is a product on each $U_i = f^{-1}((r_i - \delta_i, r_i + \delta_i)) \cong f^{-1}(r_i) \times (r_i - \delta_i, r_i + \delta_i)$.
  Then an isotopy of symplectic forms $\omega_t$, $t\in[0,1]$, such that
  \begin{equation*}
    \int_0^1 \sup_{i \in \mathbb{N}} \alpha_i \Pa{\delta_i + \Abs{I^2_{f^{-1}(r_i)}}_{\Leb^\infty}} \Abs{\omega_t^{-1}}_{\vphantom{\overset{o}{L}}\Leb^\infty(U_i, g)} \Abs{\vphantom{\omega_t^{-1}}\dot{\omega}_t}_{\vphantom{\overset{.}{L}}\Leb^\infty(U_i, g)} \der t < \infty.
  \end{equation*}
  is a strong isotopy.
\end{lemma}

\begin{proof}
  Let $J_i = (a_i, b_i) = (r_i - \delta_i, r_i + \delta_i)$ for each $i \in \mathbb{N}$.
  The assumption that $\Hml^1(f^{-1}(r),\mathbb{R}) = 0$ for all $r>R$ implies that $\dim M >2$ and $\Hml^2_\cspt(M \setminus f^{-1}(X),\mathbb{R}) \to \Hml^2(M \setminus f^{-1}(X),\mathbb{R})$ is injective for $X = \cup_{i \in \mathbb{N}}[a_i, b_i]$.
  Hence $\dim M \geq 4$ and we can apply \cref{lem:noncompact-path-method}.
  Let $X_t$ be the time-dependent vector field of \cref{lem:noncompact-path-method}.
  It suffices to show that the flow of $X_t$ starting at $t_0= 0$ exists globally for all $t \in [0, 1]$.
  Let $x \in M$.
  Fix $i_0$ such that $f(x) < a_{i_0}$.
  Let $\gamma$ be the maximal flow line of $X_t$ with $\gamma(0) = x$.
  Suppose that the maximal domain of $\gamma$ is $[0, T)$ with $0 < T \leq 1$.
  Then by \cref{lem:escape-lemma} the image $\gamma([0, T))$ must not be contained in any compact set.
  So $\lim_{t \to T} f(\gamma(t)) = \infty$.
  It follows that $\gamma$ must pass through each set $U_i$ with $i \geq i_0$.
  For each $i$ let $\ell_i = \int_{\gamma^{-1}(U_i)} \abs{X(\gamma(t))}_g\, \der t$.
  Then $\ell_i \geq \delta_i$ for each $i \geq i_0$, so $\sum_{i = i_0}^\infty \alpha_i \ell_i \geq \sum_{i = i_0}^\infty \alpha_i \delta_i = \infty$.
  
  On the other hand, by the bound on $\Abs{X_t(x)}_{\Leb^\infty(U_i, g)}$ from \cref{lem:noncompact-path-method} 
  \begin{equation*}
    \ell_i \leq \int_{\gamma^{-1}(U_i)} \Pa{\delta_i + \Abs{I^2_{f^{-1}(r_i)}}_{\Leb^\infty}} \Abs{\omega_t^{-1}}_{\vphantom{\overset{o}{L}}\Leb^\infty(U_i, g)} \Abs{\vphantom{\omega_t^{-1}}\dot{\omega}_t}_{\vphantom{\overset{.}{L}}\Leb^\infty(U_i, g)} \der t
  \end{equation*}
  and thus $\alpha_i\ell_i \leq \int_{\gamma^{-1}(U_i)} \alpha_i \Pa{\delta_i + \Abs{I^2_{f^{-1}(r_i)}}_{\Leb^\infty}} \Abs{\omega_t^{-1}}_{\vphantom{\overset{o}{L}}\Leb^\infty(U_i, g)} \Abs{\vphantom{\omega_t^{-1}}\dot{\omega}_t}_{\vphantom{\overset{.}{L}}\Leb^\infty(U_i, g)} \der t$.
  Since the subsets $\gamma^{-1}(U_i)$ of $[0, 1]$ are disjoint we have that
  \begin{equation*}
    \sum_{i = i_0}^\infty \alpha_i \ell_i \leq \int_0^1 \sup_{i \in \mathbb{N}} \alpha_i \Pa{\delta_i + \Abs{I^2_{f^{-1}(r_i)}}_{\Leb^\infty}} \Abs{\omega_t^{-1}}_{\vphantom{\overset{o}{L}}\Leb^\infty(U_i, g)} \Abs{\vphantom{\omega_t^{-1}}\dot{\omega}_t}_{\vphantom{\overset{.}{L}}\Leb^\infty(U_i, g)} \der t < \infty,
  \end{equation*}
  a contradiction.
  We conclude that $\gamma$ has domain $[0, 1]$.
\end{proof}

\begin{proof}[Proof of Main Theorem]
  Let $(M,g)$ be a Riemannian manifold with cylindrical ends and $f \colon M \to \mathbb{R}_{+}$ its radial coordinate function. Let $\omega_t$, $t \in [0, 1]$, be an isotopy of symplectic forms with total log-variation $\int_0^1 \LogVar(\omega_t,\dot{\omega}_t)\, \der t < \infty$. 
  Since $M$ is symplectic and $\Hml^1(\partial M,\mathbb{R}) = 0$, $\dim M \geq 4$. Let $\Delta$ denote the diagonal in $(1, \infty) \times (1, \infty)$.
  The finiteness of the total log-variation is equivalent to $\int_0^1 \sup_{(f(x), f(x'))\in \Delta} f(x)^{-1} \abs{\omega_t^{-1}(x)}_g\abs{\dot{\omega}_t(x')}_g \der t <\infty$.
  By continuity, any point $(c, c, t)\in \Delta \times [0, 1]$ has an open neighborhood $W_{c, t}$ such that
  $f(z)^{-1} \abs{\omega_s^{-1}(z)}_g \abs{\dot{\omega}_s(z')}_g\tenoreleven{$ is less than $}{<}\sup_{(f(x), f(x'))\in \Delta} f(x)^{-1} \abs{\omega_s^{-1}(x)}_g \abs{\dot{\omega}_s(x')}_g + 1$ 
  for all $z$, $z'$ and $s$ with $(f(z), f(z'), s)\in W_{c, t}$.
  So by the compactness of $[0, 1]$ there exists a neighborhood $W$ of $\Delta \subset (1, \infty) \times (1, \infty)$ such that
  \begin{equation}\label{eq:neighborhoodW}
    \int_0^1 \sup_{(f(x), f(x'))\in W} f(x)^{-1} \abs{\omega_t^{-1}(x)}_g\abs{\dot{\omega}_t(x')}_g \der t <\infty.
  \end{equation}
  Let $\mu \colon (1, \infty) \to \mathbb{R}_+$ be a continuous function such that $r \mapsto r + \mu(r)$ is strictly increasing and $\Set{(r, s) \in \mathbb{R}_+^2 \mmid -\mu(r) < r - s < \mu(s)} \subset W.$
  We construct disjoint subintervals $\Set{(r_i - \delta_i, r_i + \delta_i)}_{i \in \mathbb{N}}$ in $(1, \infty)$ as follows.
   Let $\delta_1 = \min\Set{1, \mu(2)/2}$ and $r_1 = 2 + \delta_1$.
  Then inductively let $\delta_{i+1} =  \min\{1,\allowbreak \mu(r_i + \delta_i)/2\}$ and $r_{i+1} = r_i + \delta_i + \delta_{i+1}$ for all $i \in\mathbb{N}$.
  We have $r_i \to \infty$ as $i \to \infty$, since otherwise the sequence $r_i$ would converge to some point $r_\infty$ with $\mu(r_\infty) = 0$.

  For each $i \in \mathbb{N}$, let $\alpha_i=1/(r_i + \delta_i)$.
  Then
  \begin{equation*}
    \sum_{i =1}^\infty 2\alpha_i\delta_i = \sum_{i =1}^\infty \Pa{1 - \frac{r_i - \delta_i}{r_{i+1} - \delta_{i+1}}} 
    \geq \sum_{i =1}^\infty \min\Set{\frac12, \frac12 \log \Pa{\frac{r_{i+1} - \delta_{i+1}}{r_i - \delta_i}}}.
  \end{equation*}
  In the last sum there are either infinitely many $i$ for which the $i$\textsuperscript{th} summand is $\frac12$, or there is some fixed $i_0 \in \mathbb{N}$ such that $i$\textsuperscript{th} summand is $\frac12 \log \Pa{\frac{r_{i+1} - \delta_{i+1}}{r_i - \delta_i}}$ for all $i\geq i_0$.
  In either case the sum diverges.
  So $\sum_{i =1}^\infty \alpha_i\delta_i = \infty$.
  By reducing each $\delta_i$ a little bit we can ensure that $r_i + \delta_i < r_{i+1} - \delta_{i+1}$ with $\sum_{i =1}^\infty \alpha_i \delta_i$ still being $\infty$.
  For each $i \in \mathbb{N}$ let $J_i = (r_i - \delta_i, r_i + \delta_i)$ and $U_i = f^{-1}(J_i)$.
  Note that \cref{eq:neighborhoodW} holds with $W$ replaced by the subset $\cup_{i \in \mathbb{N}} J_i \times J_i$, which implies $\int_0^1 \sup_{i \in \mathbb{N}} \alpha_i \Abs{\omega_t^{-1}}_{\vphantom{\overset{o}{L}}\Leb^\infty(U_i, g)} \Abs{\vphantom{\omega_t^{-1}}\dot{\omega}_t}_{\vphantom{\overset{.}{L}}\Leb^\infty(U_i, g)} < \infty$, since $\alpha_i \leq f(x)^{-1}$ for $x \in U_i$, $i \in \mathbb{N}$.
  Now since the hypersurfaces $f^{-1}(r_i)$ are all isometric to $f^{-1}(r_1)$, the quantity $\Abs{I^2_{f^{-1}(r_i)}}_{\Leb^\infty}$ is independent of $i$, and since $\delta_i\leq 1$ for all $i$ we have
  \begin{align*}
    \int_0^1 \sup_{i \in \mathbb{N}} \alpha_i \Pa{\delta_i + \Abs{I^2_{f^{-1}(r_i)}}_{\Leb^\infty}} \Abs{\omega_t^{-1}}_{\vphantom{\overset{o}{L}}\Leb^\infty(U_i, g)} \Abs{\vphantom{\omega_t^{-1}}\dot{\omega}_t}_{\vphantom{\overset{.}{L}}\Leb^\infty(U_i, g)} \der t \qquad \\
    \qquad \leq \Pa{1+\Abs{I^2_{f^{-1}(r_1)}}_{\Leb^\infty}}\int_0^1 \sup_{i \in \mathbb{N}} \alpha_i \Abs{\omega_t^{-1}}_{\vphantom{\overset{o}{L}}\Leb^\infty(U_i, g)} \Abs{\vphantom{\omega_t^{-1}}\dot{\omega}_t}_{\vphantom{\overset{.}{L}}\Leb^\infty(U_i, g)}.
  \end{align*}
  The result then follows from \cref{lem:moser-noncompact-exhausted}.
\end{proof}


\begin{proof}[Proof of \cref{cor:moser-linear-family}]
  Suppose $A = \sup_{r \in f(M)} \Abs{\omega^{-1}}_r \Abs{\vphantom{\omega^{-1}} \der \sigma}_r < 1$.
  For any $x \in M$, $\omega^{-1}(x) \der \sigma(x)$ is an endomorphism of $(\Tg_xM, \abs{\,\cdot\,}_g)$ with operator norm at most $A$.
  If $t < A^{-1}$, then $\abs{t\omega^{-1}(x) \der \sigma(x)}_g < 1$ for any $x \in M$, which means $1 + t\omega^{-1}(x) \der \sigma(x)$ is invertible.
  So $\omega_t = \omega + t \der \sigma$ is symplectic for all $t \in [0, 1]$.
  Moreover, for any $x \in M$
  \begin{equation*}
    \abs{\omega_t^{-1}(x)}_g \leq \abs{(1 + t\omega^{-1} \der \sigma)^{-1}(x)}_g \abs{\omega^{-1}(x)}_g \leq (1-tA)^{-1} \abs{\omega^{-1}(x)}_g.
  \end{equation*}
  Thus by assumption we have
  \begin{equation*}
     \int_0^1 \LogVar(\omega_t,\dot{\omega}_t)\, \der t \leq \int_0^1 \sup_{r\geq 1} (1-tA)^{-1}\Abs{\omega^{-1}}_r \Abs{\vphantom{\omega^{-1}}\der \sigma}_r  \der t \leq \frac{A}{1-A} < \infty.
  \end{equation*}
\end{proof}

\begin{proof}[Proof of \cref{cor:symplecic-compact-punctured}]
Note that \cref{cor:moser-euclidean} generalizes trivially to manifolds equipped with a metric which is Euclidean on the end(s). We will make use of this generalization, rather than arguing directly from the Main Theorem, because it makes the coordinate computations easier.
It suffices to treat the case where $F$ contains just one point $p$.
Let $g$ be a metric on $M$, and let $U$ be geodesic ball about $p$. Scaling $g$ if necessary we may take $U$ to be a unit geodesic ball, and we may use normal (exponential) coordinates to identify $(U,p)$ with $(B^{2n},0)$ where $B^{2n}$ is the unit ball in $\mathbb{R}^{2n}$. Let $\phi:U\setminus \{p\}\to \mathbb{R}^{2n}\setminus \overline{B^{2n}}$ be the diffeomorphism which in normal coordinates sends $x\in B^{2n}\setminus \{0\}$ to $\frac{x}{|x|^2}$. Under $\phi$ the radial coordinate $r$ on $\mathbb{R}^{2n}\setminus \overline{B^{2n}}$ pulls back to the reciprocal of the geodesic distance from $p$ on $U\setminus \{p\}$. Let $(x_i)$ denote the standard coordinates on $B^{2n}$ and $(\bar{x}_i)$ those on $\mathbb{R}^{2n}\setminus \overline{B^{2n}}$. Then $\phi_*\der x_i = \sum_{i=1}^{2n} (\frac{\delta_{ij}}{\abs{\bar{x}}^2}+2\frac{\bar{x}_i\bar{x}_j}{\abs{\bar{x}}^4})\der \bar{x}_j$ and $\phi_*\partial_{x_i} = \sum_{i=1}^{2n}(\abs{\bar{x}}^2\delta_{ij}+2\bar{x}_i\bar{x}_j)\partial_{\bar{x}_j}$.
Since $\dot{\omega}_t$ is bounded with respect to $g$, uniformly in $t$, the corresponding forms $\dot{\bar{\omega}}_t=\phi_*\dot{\omega}_t$ on $\mathbb{R}^{2n}\setminus \overline{B^{2n}}$ are $O(r^{-4})$, uniformly in $t$.
Similarly, from the differential of $\phi$ one has that $\bar{\omega}_t^{-1}=\phi_*\omega_t^{-1}$ is $O(r^4)$ uniformly in $t$. Pulling the Euclidean metric on $\mathbb{R}^{2n}\setminus \overline{B^{2n}}$ back to $U\setminus\{p\}$ and extending this to a metric $g'$ on $M\setminus F$ we may apply (a trivial generalization of) \cref{cor:moser-euclidean} to conclude that $\omega_t$ is a strong isotopy on $M \setminus F$.
\end{proof}
\section{Examples}

\begin{example} \label{exm:example-product}
  Consider $\mathbb{R}^{2n}$, $2n\geq 4$, with coordinates $(x_1, y_1, \dotsc, x_n, y_n)$. 
  Let $U$ be an open subset of $\mathbb{R}^{2n}$. 
  Let $f_i \in \Cont^\infty(U)$, for $i=1, \dotsc, n$. 
  Then $\omega = \sum_{i=1}^n f_i \der x_i \wedge \der y_i$ is a symplectic form if and only if each of the $f_i$ is nowhere vanishing and depends only on the coordinates $x_i$ and $y_i$. 
  The isotopy of symplectic forms $\omega_t = \sum_{i =1}^n f_i(t, x_i, y_i) \der x_i \wedge\der y_i $, $t \in [0, 1]$,  satisfies the assumption of \cref{cor:moser-euclidean} if the functions $f_i$ are bounded away from zero and have bounded time derivative.  
  Suppose $a_i \in \mathbb{R} \setminus \Set{0}$. 
  Consider the symplectic forms $\omega_t = a_1 \sqrt{x_1^2 + y_1^2 + 1 + t^2} \, \der x_1 \wedge \der y_1 + \sum_{i=2}^n a_i \der x_i \wedge \der y_i$, $t\in[0,1]$.
  By \cref{cor:moser-euclidean} there is a smooth path of diffeomorphisms $\varphi_t$ of $\mathbb{R}^{2n}$, $t \in [0, 1]$, such that $\varphi_t^* \omega_t = \omega_0$.
\end{example}

\begin{example} \label{exm:example-unbounded}
  Here we apply our result to an isotopy $\omega_t$, $t \in [0, 1]$, for which the norm of the derivative grows with $r$, while the norm of the inverse decays.
  Let $\phi \colon [0, +\infty) \to [0, +\infty)$ be a diffeomorphism such that $\Res{\phi}_{[0, 1)} = \identity$, and $\phi(r)/r$ is increasing.
  Then $\hat \phi \colon \mathbb{R}^4 \to \mathbb{R}^4$, $\hat \phi(x) = \frac{\phi(\abs{x})}{\abs{x}} x$ is a diffeomorphism.
  If $\omega = \hat \phi^* \omega_0$, then with $r = \abs{x}$ we have
  \begin{align*}
    \omega(x_1, \dotsc, x_4) 
    &= \Pa{A + B(x_1^2 + x_2^2)} \der x_1 \wedge \der x_2 + \Pa{A + B(x_3^2 + x_4^2)} \der x_3 \wedge \der x_4 \\
    &\phantom{{}=} - B(x_1 x_4 - x_2 x_3) \Pa{\der x_1 \wedge \der x_3 + \der x_2 \wedge \der x_4} \\
    &\phantom{{}=} + B(x_1 x_3 + x_2 x_4) \Pa{\der x_1 \wedge \der x_4 - \der x_2 \wedge \der x_3},
  \end{align*}
  where $A = \Pa{\frac{\phi(r)}r}^2$ and $B = \frac{\phi(r)}{r^2} \Pa{\frac{\phi(r)}r}' \geq 0$. 
  Let us fix $p > 1$,  $c \in (0, 1)$ and define $\phi$ by $\phi(r) = r^p$ for $r \geq 1$.
  Since we want $\phi$ to be smooth, we should perturb it in a neighborhood of $r=1$.
  None of our estimates are affected if this perturbation is sufficiently small, so we proceed as if $\phi$ were given by the exact formula.
  Then for $r \geq 1$ we have $A = r^{2p-2}$, $B = (p-1) r^{2p-4}$, and 
  \begin{align*}
    &\phantom{{}=} r^{4-2p} \omega(x_1, \dotsc, x_4) \\
    &= \Pa{px_1^2 + px_2^2 + x_3^2 + x_4^2} \der x_1 \wedge \der x_2 + \Pa{x_1^2 + x_2^2 + px_3^2 + px_4^2} \der x_3 \wedge \der x_4 \\
    &\phantom{{}=} - (p-1)(x_1 x_4 - x_2 x_3) \Pa{\der x_1 \wedge \der x_3 + \der x_2 \wedge \der x_4} \\
    &\phantom{{}=} + (p-1)(x_1 x_3 + x_2 x_4) \Pa{\der x_1 \wedge \der x_4 - \der x_2 \wedge \der x_3}.
  \end{align*}
  Let $\lambda \colon [0, +\infty) \to [0, +\infty)$ be an increasing smooth function which vanishes on $[0, \frac12]$, equals $1$ in $[1, +\infty)$, and satisfies $\lambda' \leq 3$.
  Let 
  \begin{equation*}
    \sigma = \frac{cp}{6(2p-1)^2} \lambda(r) r^{2p-1} \Pa{\der x_1 +\der x_2 + \der x_3 + \der x_4}.
  \end{equation*}
  Then $\der \sigma = \frac{cp}{6(2p-1)^2} \Pa{(2p-1) \lambda + \lambda' r} r^{2p-3} \sum_{i < j} (x_i - x_j) \der x_i \wedge \der x_j$.
  For an $m \times m$-matrix $Q = (q_{ij})$, the $\ell^1$ operator norm is $\abs{Q}_{\ell^1} = \max_{1 \leq i \leq m} \sum_{j=1}^m \abs{q_{ij}}$.
  For convenience we define $\Abs{\cdot}_r$ as the supremum over the sphere of radius $r$ of this pointwise norm (rather than of the equivalent $\ell^2$ norm). We then have $\Abs{\omega^{-1}}_r \leq  (2 - p^{-1}) r^{2-2p}$ if $r \geq 1$,
  and $\Abs{\omega^{-1}}_r = 1$ if $r < 1$.
  Similarly
  $\Abs{\vphantom{\omega^{-1}}\der \sigma}_r \leq \frac{cp}{2p-1} r^{2p-2}$ if $r \geq 1$,
  and $\abs{\vphantom{\omega_t^{-1}}\der\sigma(x)} \leq c$ if $r < 1$. 
  Since $\abs{\omega^{-1}(x)} \abs{\vphantom{\omega_t^{-1}}\der \sigma(x)} \leq c < 1$ the $2$-form $\omega_t = \omega + t \der \sigma$ is nondegenerate for every $t \in [0, 1]$ (cf. the proof of \cref{cor:moser-linear-family}). 
  Moreover, $\int_0^1 \sup_{r \geq 1} \Abs{\omega_t^{-1}}_r \Abs{\vphantom{\omega_t^{-1}}\dot{\omega}_t}_r$ is finite.
  So $\omega_t$, $t\in[0,1]$, is a strong isotopy by \cref{cor:moser-euclidean}.
\end{example}

\begin{example}\label{bad-example} Here we give an example of a strong isotopy with infinite log variation. 
  Consider the unit sphere $S^3$ contained in $\mathbb{R}^4$ with coordinates $(x_1,y_1,x_2,y_2)$, and let $\alpha_{0}=\frac{1}{2}(x_1\der y_1 - y_1 \der x_1 + x_2 \der y_2 - y_2 \der x_2)$ be the standard contact form on $S^3$. Consider the rescaled contact form $\alpha = (2x_1^2 + y_1^2)\alpha_0$ on $S^3$. The structure $(S^3,\alpha)$ can be realized as the boundary of a Liouville domain $(\Omega, \omega, V)$ in the sense of \cite{CiEl2012} (in fact this may be taken to be the boundary of a star convex domain in $\mathbb{R}^4$ with the standard symplectic form). The Liouville completion of $(\Omega, \omega, V)$ is constructed by attaching $S^3\times [0,\infty)$ to $\Omega$, where the symplectic form on $S^3\times [0,\infty)$ is $\der(e^r\alpha)$ with $r$ the coordinate on $[0,\infty)$. The resulting symplectic manifold $(M,\omega)$ is symplectomorphic to the standard $\mathbb{R}^4$, but this construction allows us to more easily write down the required family of diffeomorphisms of $M$. For $t\in \mathbb{R}$ define $\phi_t: S^3\times [0,\infty) \to S^3\times [0,\infty)$ by $(x,r)\mapsto (e^{itr^p}\cdot x, r)$, where $e^{i\theta}$ acts on $S^3$ by a rotation through angle $\theta$ in the $(x_1,y_1)$-plane. The family $\phi_t$ may be extended to a smooth $1$-parameter family of diffeomorphism of $M$, which we still denote $\phi_t$. Let $\omega_t = \phi_t^*\omega$, $t\in [0,1]$. Then on $M\setminus \Omega = S^3\times [0,\infty)$ we have $\omega_t = e^r[(1+\cos^2(tr^p))x_1^2-\sin(2tr^p)x_1y_1+(1+\sin^2(tr^p))y_1^2](\der \alpha_0+\der r \wedge \alpha) + e^r [ 2(1+\cos^2(tr^p))x_1\der x_1 - \sin(2tr^p)(x_1\der y_1 + y_1\der x_1) + 2(1+\sin^2(tr^p))y_1\der y_1 ]\wedge \alpha_0$. From this it is easy to see that $\Abs{\omega_t^{-1}}_r \sim e^{-r}$ whereas $\Abs{\vphantom{\omega_t^{-1}}\dot{\omega}_t}_r \sim r^p e^r$, so that $\Abs{\omega_t^{-1}}_r \Abs{\vphantom{\omega_t^{-1}}\dot{\omega}_t}_r \sim r^p$ and hence the Main Theorem does not apply. Although this is a path of Liouville structures by construction, it is not obvious from the formula for $\omega_t$. 
\end{example}
\section{Concluding remarks} \label{sec:concluding-remarks}


\subsection{Na\"ive symplectic stability on \texorpdfstring{$\mathbb{R}^{2n}$}{R2n}}
For $\mathbb{R}^{2n}$ it is possible to get a
naive symplectic stability result with a completely elementary
proof as follows.
Let $\omega_t$, $t \in [0, 1]$, be \tenoreleven{a symplectic isotopy}{an isotopy of symplectic forms} on $\mathbb{R}^{2n}$ with $\int_0^1 \sup_{x \in \mathbb{R},\; s \in [0, 1]} s\abs{x} \abs{\omega_t^{-1}(x)}_\Euclid \abs{\dot{\omega}_t(sx)}_\Euclid \der t$ finite.
Then $\omega_t$ is a strong isotopy.
To verify this let $E$ be the Euler vector field on $\mathbb{R}^{2n}$ and $I \colon \Omega^2(\mathbb{R}^{2n}) \to \Omega^1(\mathbb{R}^{2n})$ be given by $I \omega(x) = \int_0^1 E(sx) \intprod \omega(sx)\, \der s$. 
Then $\der I \omega = \omega$ for any exact $2$-form $\omega$.
Let $\sigma_t = I\dot{\omega}_t$ and let $X_t = -\omega_t^{-1} \sigma_t$.
Let $x \in \mathbb{R}^{2n}$ and let $\gamma$ be the maximal flow line of $X_t$ with $\gamma(0) = x$.
If the maximal domain of $\gamma$ is $[0, T)$ with $0 < T \leq 1$, then by \cref{lem:escape-lemma} the image of $\gamma$ must not be contained in any compact set.
But the length of $\gamma$ is bounded by $\int_0^T \sup_{x \in \mathbb{R},\; s \in [0, 1]} s\abs{x} \abs{\omega_t^{-1}(x)}_\Euclid \abs{\dot{\omega}_t(sx)}_\Euclid \der t <\infty$ so $\gamma([0, T))$ is precompact.
So the flow $\varphi_t$ of $X_t$ starting from $t_0 = 0$ exists for all $t \in [0, 1]$.

\subsection{Symplectic stability for compactly supported isotopies}
Using Lemma \ref{lem:primitive-form-compact-support} one can generalize Moser's stability theorem to apply to compactly supported isotopies: 
Let $\omega_t$, $t \in [0, 1]$, be an isotopy of symplectic forms on a manifold $M$ such that $\support(\omega_t - \omega_0) \subset W$ for all $t$, where $W \subset M$  is an open submanifold with compact closure and smooth boundary, and the cohomology class of $\Res{(\omega_t - \omega_0)}_W$ in $\Hml_\cspt^2(W,\mathbb{R})$ is trivial for all $t$. Then for any smoothly bounded precompact open submanifold $V$ of $M$ with $\overline{W} \subset V$ there exists a smooth path of diffeomorphisms of $M$ fixing $M\setminus V$ such that $\varphi_t^* \omega_t = \omega_0$ for all $t$.
Indeed, let $I^2_{M, V}$ be as in \cref{lem:primitive-form-compact-support}.
Let $\sigma_t = I^2_{M, V} \dot{\omega}_t$, then $\der \sigma_t = \dot{\omega}_t$.
Then $X_t = -\omega_t^{-1} \sigma_t$ is compactly supported in $W$ and therefore complete; the flow of $X_t$ fixes points in $M \setminus V$.
By the Path Method the flow $\varphi_t$ of $X_t$ satisfies $\varphi_t^* \omega_t = \omega_0$ for all $t$.
In fact, the result holds for $W$ any precompact open set, cf. \cite[Theorem~6.8]{CiEl2012} or \cite[Lemma, page 617]{MR728456} for alternative approaches (we chose to keep with the Hodge theoretic approach in establishing \cref{lem:primitive-form-compact-support}). 

This result was used in the proof of the stability result \cite[Proposition 11.8]{CiEl2012} for ``Liouville homotopies'' of Liouville manifolds, where it plays a role analogous to our use of \cref{lem:primitive-form-compact-support} on the gluing regions: By assuming the existence of smoothly varying families of compact hypersurfaces transverse to the (radial) Liouville vector field Cieliebak and Eliashberg are able to construct the required $1$-parameter family of diffeomorphisms on certain primary regions by applying Gray's theorem \cite{Gray59} to these hypersurfaces and then using the local product structure coming from the Liouville vector field; the resulting $1$-parameter family of diffeomorphisms can be fixed up on the remaining gluing regions by using the above generalization of Moser's theorem. Without the convexity assumptions on the symplectic forms, however, and the compatible ``Liouville homotopy'' giving the smooth families of contact hypersurfaces on which one can apply Gray's theorem, the generator $X_t$ for the strong symplectic isotopy one is trying to construct needs to be estimated to determine its integrability.

\subsection{Punctured compact manifolds}
Considering punctured compact manifolds allows for a comparison of sorts between our result and the original result of Moser. 
\cref{cor:symplecic-compact-punctured} states that a symplectic isotopy on a punctured compact manifold $M\setminus F$ such that $\omega_t^{-1}$ and $\dot{\omega}_t$ are uniformly bounded with respect to a metric defined on $M$ is a strong isotopy, provided $\mathrm{dim}M\geq 4$. 
Slightly modifying Moser's proof in the compact case one has a direct elementary proof of the weaker result:
Let $M$ be a compact manifold and let $F$ be a finite set of points on $M$.
If $\omega_t$, $t\in [0,1]$, is a symplectic isotopy on $M \setminus F$ which is the restriction of a symplectic isotopy on $M$, then $\omega_t$ is a strong isotopy on $M \setminus F$.
To demonstrate this let $\omega_t$ also denote the symplectic isotopy on $M$ whose restriction is the isotopy $\omega_t$ on $M\setminus F$. 
Construct $X_t$ on $M$ as in the usual proof of Moser's theorem. 
Since $F$ is finite, for each $t$ one can choose a Hamiltonian vector field $Y_t$ (Hamiltonian with respect to $\omega_t$) for which $Y_t|_F = - X_t|_F$. Since $X_t$ is smooth in $t$, $Y_t$ can be chosen smooth in $t$. By the usual argument the flow $\varphi_t$, $t\in [0,1]$, generated by $X_t+Y_t$ satisfies $\varphi_0=\mathrm{id}$ and $\varphi_t^*\omega_t=\omega_0$. Moreover, by construction $\varphi_t$ preserves $F$. So $\varphi_t|_{M\setminus F}$ is the required strong isotopy.

\subsection{Contact stability}
The previous ideas apply trivially to contact manifolds.
Let $(M, g)$ be a complete oriented odd dimensional Riemannian manifold.
Let $\theta_t$, $t \in [0, 1]$, be a smooth path of contact forms on $M$ with $\int_0^1 \sup_M \big|(\der\theta_t|_{H_t})^{-1}\dot{\theta}_t|_{H_t}\big|_g \der t <\infty$ where $H_t = \ker \theta_t$.
Then there exists a smooth path $\varphi_t$ of diffeomorphisms of $M$ and $f_t$ of positive smooth functions on $M$ such that $\varphi_0 = \identity$ and $\varphi_t^* \theta_t = f_t\theta_0$ for $t \in [0, 1]$. 
Indeed, this case is easy because one does not need to invert the exterior derivative to construct the time-dependent vector field (using the `path method' of Gray \cite{Gray59}).
Let $H_t=\ker \theta_t$, and let $H= H_0$. Let $X_t$ be the time dependent vector field $-(\der\theta_t|_{H_t})^{-1}(\dot{\theta}_t|_{H_t})$. 
Let $x \in M$ and let $\gamma$ be the maximal flow line of $X_t$ with $\gamma(0) = x$.
If the maximal domain of $\gamma$ is $[0, T)$ with $0 < T \leq 1$, then by \cref{lem:escape-lemma} the image of $\gamma$ must not be contained in any compact set. 
However, the length of $\gamma$ is bounded by $\int_0^1 \sup_M \abs{(\der\theta_t|_{H_t})^{-1}\dot{\theta}_t|_{H_t}}_g \der t$ and therefore $\gamma([0, T))$ is precompact. 
So the flow $\varphi_t$ of $X_t$ starting from $t_0 = 0$ exists for all $t \in [0, 1]$. 
Let $R_t$ denote the Reeb vector field of $\theta_t$ and let $h_t = \dot{\theta}_t(R_t)$. 
We compute, using Cartan's formula and $\theta_t(X_t) = 0$,
\begin{equation*}
  \frac{\der}{\der t}(\varphi_t^* \theta_t) = \varphi_t^*(\mathcal{L}_{X_t} \theta_t + \dot{\theta_t}) = \varphi_t^*(-\dot{\theta_t}|_{H_t} + \dot{\theta_t}) = \varphi_t^*(\dot{\theta_t}(R_t)\theta_t) = h_t \varphi_t^* \theta_t.
\end{equation*}
Since $\varphi_0^* \theta_0 = \theta_0$ there exists $f_t$ such that $\varphi_t^* \theta_t = f_t\theta_0$ for all $t \in [0, 1]$.

\bibliographystyle{abbrv}
\bibliography{ref}

\authaddresses

\end{document}